\documentclass[11pt]{amsart}
\usepackage{graphicx}
\usepackage{float}

\usepackage{url}
\usepackage{setspace}
\usepackage{tikz}
\DeclareMathAlphabet{\mathcal}{OMS}{cmsy}{m}{n}
\usepackage{mathrsfs}
\usepackage{indentfirst,csquotes}

\usepackage{amssymb,amsthm,amsmath}
\usepackage{listings}
\usepackage{xcolor}
\usepackage{xcolor,paralist,titlesec,fancyhdr,etoolbox}
\usepackage[
    colorlinks=true,
    linkcolor=blue,
    urlcolor=blue,
    citecolor=blue
]{hyperref}
\usepackage{doi}
\usepackage{amsthm}
\usepackage{orcidlink}

\theoremstyle{definition}
\newtheorem{theorem}{Theorem}[section]
\numberwithin{theorem}{section}
\newtheorem{definition}[theorem]{Definition}

\newtheorem{lemma}[theorem]{Lemma}
\newtheorem{proposition}[theorem]{Proposition}
\newtheorem{corollary}[theorem]{Corollary}

\titleformat{\section}[display]{\normalfont\huge\bfseries\centering}{\centering\chaptertitlename\thechapter}{10pt}{\Large}
\titlespacing*{\section}{0pt}{0ex}{0ex}

\titleformat{\section}
  {\normalfont\Small\centering} % Font, size, and centering
  {\thesection.}{1em}{\MakeUppercase} % Automatic uppercase

\titlespacing*{\section}{0pt}{1.5ex plus 1ex minus .2ex}{1.3ex plus .2ex}

\usepackage{enumerate}
\usepackage{enumitem}
\title[A Proof of Ramanujan's Classic $\pi$ Formula]% end with percent
 {A Proof of Ramanujan's Classic $\pi$ Formula} 
 \author{Thang Pang Ern \orcidlink{0009-0005-3447-4188}}
\address{Department of Mathematics, National University of Singapore, 10 Lower Kent Ridge Road, Singapore 119076}
\email{thangpangern@u.nus.edu}

\author{Devandhira Wijaya Wangsa}
\address{Department of Mathematics, The Hong Kong University of Science and Technology, Clearwater Bay, Kowloon, Hong Kong}
\email{dwwangsa@connect.ust.hk}
\begin{document}
\maketitle
\begin{abstract}
In 1914, Ramanujan presented a collection of 17 elegant and rapidly converging formulae for $\pi$. Among these, one of the most celebrated is the following series: \begin{align*}
    \frac{1}{\pi}=\frac{2\sqrt{2}}{9801}\sum_{n=0}^{\infty}\frac{26390n+1103}{\left(n!\right)^4} \frac{\left(4n\right)!}{396^{4n}}.
\end{align*}
In this paper, we give a full proof of this classic formula using hypergeometric series and a special type of lattice sum due to Zucker and Robertson. We will also use some results by Dirichlet and Edwards in algebraic number theory.
\end{abstract}
\section{Introduction}
In 1914, Ramanujan provided a list of 17 formulae for $\pi$ \cite{ramanujan modular} without proofs. These identities are remarkable not only for their elegance, but also for their computational efficiency: truncating after only a modest number of terms already yields many correct digits of $\pi$. In fact, in November 1985, R. W. Gosper, Jr. used (\ref{eqn: what to prove ramanujan}) to calculate 17,526,100 digits of $\pi$, which at that time was a world record \cite{baruahberndtchan2009}. In 1987, the Borwein brothers gave proofs of all of Ramanujan's $\pi$ formulae \cite{pi and the agm, borwein pi proof 1987}, including the one of interest in this paper, which is \begin{align}\label{eqn: what to prove ramanujan}
\frac{1}{\pi}=\frac{2\sqrt{2}}{9801}\sum_{n=0}^{\infty}\frac{26390n+1103}{\left(n!\right)^4} \frac{\left(4n\right)!}{396^{4n}}.
\end{align} However, the computation of the Ramanujan \( g \)-invariant \( g_{58} \) was notably absent. The \(g\)-invariant plays a critical role in deriving Ramanujan-type formulae for \(\pi\), making its calculation particularly significant. This generalisation is known as a Ramanujan-Sato series. In general, such a series is of the form \[\frac{1}{\pi}=\sum_{n=0}^{\infty}s_n \frac{An+B}{C^n},\]
where $s_n$ is a sequence of integers satisfying a recurrence relation, and $A,B,C$ are modular forms.
\section{Elliptic Integrals, Theta Functions, and Hypergeometric Series}
\noindent \begin{definition}[complete elliptic integrals of the first and second kind]\label{definition: complete elliptic integrals of the first and second kind}
Let $k\in\left[0,1\right]$ denote the elliptic modulus, which is a quantity used in the study of elliptic functions and elliptic integrals. Then, let $k'=\sqrt{1-k^2}$ be the complementary modulus. Define the complete elliptic integrals of the first and second kind, $K\left(k\right)$ and $E\left(k\right)$ respectively, to be \begin{align}\label{eqn: K and E integrals}
        K\left(k\right)=\int_{0}^{\frac{\pi}{2}}\frac{1}{\sqrt{1-k^2\operatorname{sin}^2\theta}}\text{ }d\theta\quad\text{and}\quad 
        E\left(k\right)=\int_{0}^{\frac{\pi}{2}}\sqrt{1-k^2\operatorname{sin}^2\theta}\text{ }d\theta.
    \end{align}
\end{definition}
\begin{theorem}\label{derivatives of K and E}
The derivatives of $K$ and $E$ in Definition \ref{definition: complete elliptic integrals of the first and second kind} satisfy the differential equations 
\begin{align*}
    \frac{dK}{dk}=\frac{E-\left(k'\right)^2K}{k\left(k'\right)^2}\quad\text{and}\quad 
    \frac{dE}{dk}=\frac{E-K}{k}.
\end{align*}
\end{theorem}
\begin{proof}
We will only briefly discuss the first differential equation, which follows easily by Leibniz's rule of differentiating under the integral sign. By using the formula for $K$ in (\ref{eqn: K and E integrals}), we see that \begin{align*}
 K'\left(k\right)=\int_{0}^{\frac{\pi}{2}}\frac{\partial}{\partial k}\left(\frac{1}{\sqrt{1-k^2\operatorname{sin}^2\theta}}\right)\text{ }d\theta,
\end{align*}
and the first result follows. The other differential equation can be deduced similarly.
\end{proof}
\begin{definition}\label{definition of complementary integral}
    The complementary integrals are defined as follows:
    \begin{align*}
        K'\left(k\right)=K\left(k'\right)\quad\text{and}\quad E'\left(k\right)=E\left(k'\right).
    \end{align*}
\end{definition}
From Definition \ref{definition of complementary integral}, one can find a nice relationship between $K,K',E,E'$ known as Legendre's relation \cite{Jesus}. It states that \begin{align}\label{eqn: Legendre relation}
K\left(k\right)E'\left(k\right)+E\left(k\right)K'\left(k\right)-K\left(k\right)K'\left(k\right)=\frac{\pi}{2}.
\end{align}
\begin{definition}[one-variable Jacobi theta functions] Define the Jacobi theta functions of one variable to be as follows: \begin{align*}
        \theta_2\left(q\right)=\sum_{n\in\mathbb{Z}}q^{\left(n+\frac{1}{2}\right)^2}\quad\text{and}\quad \theta_3\left(q\right)=\sum_{n\in\mathbb{Z}}q^{n^2}\quad\text{and}\quad\theta_4\left(q\right)=\sum_{n\in\mathbb{Z}}\left(-1\right)^nq^{n^2}=\theta_3\left(-q\right)
    \end{align*}
where $q$, which satisfies $\left|q\right|<1$, is called the nome of the associated theta function.
\end{definition}
We can define the nome $q$ in terms of the elliptic modulus $k$ as follows:
    \begin{align*}
        q=\operatorname{exp}\left[-\pi\cdot \frac{K'\left(k\right)}{K\left(k\right)}\right].
    \end{align*} 
It is important to see $k$ as a function of $q$. As such, we have Theorem \ref{elliptic modulus and nome relation}, which states three equations relating the elliptic modulus $k$ to the nome $q$.
\begin{theorem}\label{elliptic modulus and nome relation} We have
    \begin{align*}
        k=\frac{\theta_2^2\left(q\right)}{\theta_3^2\left(q\right)}\quad\text{and}\quad k'=\frac{\theta_4^2\left(q\right)}{\theta_3^2\left(q\right)}\quad\text{and}\quad K\left(k\right)=\frac{\pi}{2}\theta_3^2\left(q\right).
    \end{align*}
\end{theorem}
We now proceed to discuss the Ramanujan $g$-invariant or class invariant (Definition \ref{ramanujan g invariant}), which we denote by $g_n$, or simply $g$. While a closely related concept, the $G$-invariant, also exists and shares similarities with the $g$-invariant, our focus here will remain exclusively on the latter.
\begin{definition}[Ramanujan $g$-invariant]\label{ramanujan g invariant}
    Define the Ramanujan $g$-invariant to be
    \begin{align*}
        g=\left(\frac{\left(k'\right)^2}{2k}\right)^{\frac{1}{12}}.
    \end{align*}
\end{definition}
Ramanujan gave the following formula for $g_n$ as an infinite product \cite{ramanujan modular}: \begin{align}\label{eqn: useful infinite product}
    \prod_{k=1,3,5,\ldots}\left(1-e^{-k\pi\sqrt{n}}\right)=2^{1/4}e^{-\pi\sqrt{n}/24}g_n.
\end{align}
This is particularly useful for evaluating the specific case where $n=58$. Actually, it is not surprising that $g_n$ can be represented by the infinite product in (\ref{eqn: useful infinite product}) as the elliptic modulus $k$ can be expressed in terms of $g$, i.e. 
\begin{align*}
        k=g^6\sqrt{g^{12}+g^{-12}}-g^{12}.
    \end{align*}
\begin{definition}[singular value functions]
    Define \begin{align}\label{eqn: singular value function of the first kind}
        \lambda^\ast\left(r\right)=k\left(\operatorname{exp}\left(-\pi\sqrt{r}\right)\right)\quad \text{where }r>0
    \end{align}
    to be the singular value function of the first kind. Also, the singular value function of the second kind, $\alpha$, is given by the following formula:
    \begin{align}\label{eqn: definition of alpha}
        \alpha\left(r\right)=\frac{E'\left(k\right)}{K\left(k\right)}-\frac{\pi}{4\left[K\left(k\right)\right]^2}\quad\text{where }r>0.
    \end{align}
\end{definition}
\begin{theorem} Let $\alpha$ denote the singular value function of the second kind as in (\ref{eqn: definition of alpha}). Then, the following holds:
    \begin{align*}
        \lim_{r\rightarrow\infty}\alpha\left(r\right)=\frac{1}{\pi}.
    \end{align*}
\end{theorem}
\begin{proof}
    Since $\displaystyle \lim_{r\rightarrow\infty}\lambda^\ast\left(r\right)=0$, then \begin{align*}
        0<\alpha\left(r\right)-\frac{1}{\pi}\le \sqrt{r}\left[\lambda^\ast\left(r\right)\right]^2\le \frac{16\sqrt{r}}{e^{\pi\sqrt{r}}}.
    \end{align*}
    and the result follows by the squeeze theorem.
\end{proof}
In Theorem \ref{formula for alpha}, we present a different formula for $\alpha\left(r\right)$ only in terms of the two complete elliptic integrals $K$ and $E$.
\begin{theorem}\label{formula for alpha}
    \begin{align*}
        \alpha\left(r\right)=\frac{\pi}{4\left[K\left(k\right)\right]^2}-\sqrt{r}\left[\frac{E\left(k\right)}{K\left(k\right)}-1\right].
    \end{align*}
\end{theorem}
\begin{proof}
We have \begin{align}\label{eqn: long formula for alpha}
    \alpha \left( r \right)=\frac{{E}'\left( k \right)}{K\left( k \right)}-\frac{\pi }{4{{\left[ K\left( k \right) \right]}^{2}}}=\frac{4{E}'\left( k \right)K\left( k \right)-\pi }{4{{\left[ K\left( k \right) \right]}^{2}}}=\frac{4K\left( k \right){K}'\left( k \right)-4E\left( k \right){K}'\left( k \right)-\pi }{4{{\left[ K\left( k \right) \right]}^{2}}}
\end{align}
where the last equality follows from Legendre's relation (\ref{eqn: Legendre relation}). From (\ref{eqn: singular value function of the first kind}), one can deduce that \begin{align}\label{eqn: K'/K}
    \frac{K'\left(\lambda^\ast\left(r\right)\right)}{K\left(\lambda^\ast\left(r\right)\right)}=\sqrt{r}.
\end{align}
The result follows by plugging (\ref{eqn: K'/K}) into (\ref{eqn: long formula for alpha}).
\end{proof}
We shall show that $\alpha$ has a direct connection with $1/\pi$ in Theorem \ref{theorem alpha connection with 1/pi}. Also, it is a well-known result that for any positive rational number $r$, $\alpha\left(r\right)$ is an algebraic number.
\begin{theorem}\label{theorem alpha connection with 1/pi} Let $k=\lambda^\ast\left(r\right)$ denote the singular value function of the first kind. Then,
    \begin{align*}
        \frac{1}{\pi}=\sqrt{r}k\left(k'\right)^2\left[\left(\frac{2}{\pi}\right)^2 K\left(k\right)\frac{dK}{dk}\right]+\left[\alpha\left(r\right)-\sqrt{r}k^2\right]\left[\frac{2}{\pi}K\left(k\right)\right]^2.
    \end{align*}
\end{theorem}

\begin{definition}[hypergeometric series]
    The Gaussian hypergeometric function $_{2}F_1$ is given by \begin{align*}
        _2F_1\left(a,b;c;z\right)=\sum_{n=0}^{\infty}\frac{\left(a\right)_n\left(b\right)_n}{\left(c\right)_n} \frac{z^n}{n!}\quad\text{where }a,b,c \in\mathbb{C}.
    \end{align*}
    We also define $_3F_2$ to be 
    \begin{align*}
        _3F_2\left(a,b,c; \alpha, \beta ;z\right)=\sum_{n=0}^{\infty}\frac{\left(a\right)_n\left(b\right)_n \left(c\right)_n}{\left(\alpha \right)_n\left(\beta \right)_n}\frac{z^n}{n!}\quad\text{where }a,b,c,\alpha,\beta \in\mathbb{C}.
    \end{align*}
\end{definition}
\begin{proposition}
    For $k\in \left[0,\frac{1}{\sqrt{2}}\right]$, we have the following identities: \begin{align*}
        \frac{2}{\pi}K\left(k\right)={_2F_1}\left(\frac{1}{4},\frac{1}{4};1;\left(2kk'\right)^2\right)\quad\text{and}\quad 
        \left[\frac{2}{\pi}K\left(k\right)\right]^2={_3F_2}\left(\frac{1}{2},\frac{1}{2},\frac{1}{2};1,1;\left(2kk'\right)^2\right).
    \end{align*}
\end{proposition}
\begin{proof}
For the first identity, recall Kummer's identity, which states that 
\begin{align*}
    {_2F_1}\left(2a,2b;a+b+\frac{1}{2};z\right)={_2F_1}\left(a,b;a+b+\frac{1}{2};4z\left(1-z\right)\right).
\end{align*}
Set $z=k^2$, so $1-z=1-k^2=\left(k'\right)^2$. Also, set $a=b=1/4$, so we obtain the following (we will be working with the $q$-shifted factorial here although it will be formally defined in Definition \ref{q shifted factorial}): \begin{align*}
{_2F_1}\left(\frac{1}{4},\frac{1}{4};1;\left(2kk'\right)^2\right)={_2F_1}\left(\frac{1}{2},\frac{1}{2};1;k^2\right)=\sum_{n=0}^{\infty}\frac{\left(\frac{1}{2}\right)_n\left(\frac{1}{2}\right)_n}{\left(1\right)_n}\cdot \frac{k^{2n}}{n!}=\sum_{n=0}^{\infty}\frac{\left(\left(2n\right)!\right)^2}{16^n\left(n!\right)^4}\cdot k^{2n}
\end{align*}
By considering the series expansion of $K$ in (\ref{eqn: K and E integrals}), one can deduce that the first identity holds.

For the second identity, we use the following identity by Clausen \cite{Jesus}: \begin{align*}
\left({_2F_1}\left(a,b;a+b+\frac{1}{2};z\right)\right)^2={_3F_2}\left(2a,a+b,2b;a+b+\frac{1}{2},2a+2b;z\right)
\end{align*}
Again, we set $z=\left(2kk'\right)^2$, $a=b=1/4$, which yields the result.
\end{proof}
\begin{corollary}\label{seriesforKintermsofhypergeometric}
For $0\le k\le \frac{1}{\sqrt{2}}$, we have the following:
\begin{align}\label{eqn: 2/pi K(k) ^2}
    \left[\frac{2}{\pi}K\left(k\right)\right]^2=\frac{1}{k^2}\cdot {_3F_2}\left(\frac{1}{4},\frac{1}{2},\frac{3}{4};1,1;\left(\frac{2}{g^{12}+g^{-12}}\right)^2\right)
\end{align}
\end{corollary}
We have provided a series for $\left[\frac{2}{\pi}K\left(k\right)\right]^2$ in terms of Ramanujan's $g$-invariant. We see that the formula obtained in Corollary \ref{seriesforKintermsofhypergeometric} is of the form \begin{align*}
    \left[\frac{2}{\pi}K\left(k\right)\right]^2=m\left(k\right)F\left(\varphi\left(k\right)\right)=mF,
\end{align*}
where $m$ and $\varphi$ are algebraic numbers, and $F\left(\varphi\right)$ is a hypergeometric series which is given by the expansion 
\begin{align*}
    F\left(\varphi\left(k\right)\right)=\sum_{n=0}^{\infty}a_n\varphi^n\quad\text{where }a_n\in\mathbb{Q}\text{ for all }n\in\mathbb{N}.
\end{align*}
Differentiating both sides of (\ref{eqn: 2/pi K(k) ^2}), we obtain 
\begin{align*}
    \left(\frac{2}{\pi}\right)^2 K\frac{dK}{dk}=\frac{1}{2}\left(\frac{dm}{dk}F+m\frac{d\varphi}{dk}\frac{dF}{d\varphi}\right).
\end{align*}
Note that 
\begin{align*}
    \frac{dF}{d\varphi}=\sum_{n=1}^{\infty}na_n\varphi^{n-1}=\sum_{n=0}^{\infty}\left(n+1\right)a_{n+1}\varphi^n=\frac{1}{\varphi}\sum_{n=0}^\infty \left(n+1\right)a_{n+1}\varphi^{n+1}
\end{align*}
By Theorem \ref{theorem alpha connection with 1/pi}, this yields 
\begin{align*}
  \frac{1}{\pi }&=\sqrt{r}k{{\left( {{k}'} \right)}^{2}}\left[ \frac{1}{2}\left( \frac{dm}{dk}F+m\frac{d\varphi }{dk}\frac{dF}{d\varphi } \right) \right]+\left[ \alpha \left( r \right)-\sqrt{r}{{k}^{2}} \right]mF \\ 
 & =\sqrt{r}k{{\left( {{k}'} \right)}^{2}}\frac{F}{2}\frac{dm}{dk}+\sqrt{r}k{{\left( {{k}'} \right)}^{2}}\frac{m}{2}\frac{d\varphi }{dk}\frac{dF}{d\varphi } +\left[ \alpha \left( r \right)-\sqrt{r}{{k}^{2}} \right]mF 
\end{align*}
so \begin{align}\label{eqn: 1/pi brace}
\frac{1}{\pi}=\sum\limits_{n=0}^{\infty }{{{a}_{n}}\left\{ \frac{1}{2}\sqrt{r}k{{\left( {{k}'} \right)}^{2}}\frac{dm}{dk}+\left[ \alpha \left( r \right)-\sqrt{r}{{k}^{2}} \right]m+\frac{mn}{2\varphi }\sqrt{r}k{{\left( {{k}'} \right)}^{2}}\frac{d\varphi }{dk} \right\}{{\varphi }^{n}}}.
\end{align}
We see that the braced term of in the expansion of $1/\pi$ in (\ref{eqn: 1/pi brace}) is of the form $A+nB$, where $A$ and $B$ are algebraic numbers.

By setting \begin{align}\label{eqn: x_N formula}
    x_N=\frac{2}{g_N^{12}+g_{N}^{-12}}=\frac{4k_N\left(k_N'\right)^2}{\left(1+k_N^2\right)^2},
\end{align}
we deduce the following series in $x_N$, which also appears in \cite{pi and the agm}:
\begin{align}\label{eqn: 1/pi in terms of everything}
    \frac{1}{\pi}=\sum_{n=0}^{\infty}\frac{\left(\frac{1}{4}\right)_n\left(\frac{1}{2}\right)_n\left(\frac{3}{4}\right)_n}{\left(n!\right)^3}\left[\frac{\alpha\left(N\right)}{x_N\left(1+k_N^2\right)}-\frac{\sqrt{N}}{4g_N^{12}}+n\sqrt{N}\cdot \frac{g_N^{12}-g_N^{-12}}{2}\right]x_{N}^{2n+1}
\end{align}
In (\ref{eqn: 1/pi in terms of everything}), we have the quantities $\left(\frac{1}{4}\right)_n$, $\left(\frac{1}{2}\right)_n$, and $\left(\frac{3}{4}\right)_n$. We shall obtain alternative expressions for these in terms of more familiar-looking ones in Lemma \ref{lemma hypergeometric}. In the study of hypergeometric series, expressions like these are said to be defined by the $q$-shifted factorial (Definition \ref{q shifted factorial}).
\begin{definition}[$q$-shifted factorial]\label{q shifted factorial}
Define \begin{align*}
    \left(q\right)_n=\begin{cases}
        1 & \text{if }n=0;\\
        q\left(q+1\right)\cdots\left(q+n-1\right) & \text{if }n>0.
    \end{cases}
\end{align*}
\end{definition}
\begin{lemma}\label{lemma hypergeometric}
\begin{align*}
    \left(\frac{1}{4}\right)_n\left(\frac{1}{2}\right)_n\left(\frac{3}{4}\right)_n=\frac{1}{256^n} \frac{\left(4n\right)!}{n!}.
\end{align*}
\end{lemma}
\begin{proof}
We have \begin{align*}
    \left(\frac{1}{4}\right)_n&=\left(\frac{1}{4}\right)\left(\frac{5}{4}\right)\left(\frac{9}{4}\right)\cdots \left(\frac{4n-3}{4}\right)=\frac{1\cdot 5 \cdot 9 \cdots \left(4n-3\right)}{4^n}
\end{align*}
and \begin{align*}
    \left(\frac{1}{2}\right)_n=\left(\frac{1}{2}\right)\left(\frac{3}{2}\right)\left(\frac{5}{2}\right)\cdots\left(\frac{2n-1}{2}\right)=\frac{1\cdot 3\cdot 5\cdots  \left(2n-1\right)}{2^n}
\end{align*}
as well as \begin{align*}
\left(\frac{3}{4}\right)_n=\left(\frac{3}{4}\right)\left(\frac{7}{4}\right)\left(\frac{11}{4}\right)\cdots \left(\frac{4n-1}{4}\right)=\frac{3\cdot 7 \cdot 11 \cdots\left(4n-1\right)}{4^n}.
\end{align*}
Putting everything together, we have
\begin{align*}
\left(\frac{1}{4}\right)_n\left(\frac{1}{2}\right)_n\left(\frac{3}{4}\right)_n=\frac{1}{32^n} \frac{\left(4n\right)!}{4^n\left(2n\right)!} \frac{\left(2n\right)!}{2^nn!}=\frac{1}{256^n} \frac{\left(4n\right)!}{n!}.
\end{align*}
\end{proof}
\section{Computation of the Ramanujan $g$-invariant}
One of the most difficult parts of deducing (\ref{eqn: what to prove ramanujan}) is computing the exact value of the $g$-invariant $g_{58}$. Once we deduce it, we can plug it into (\ref{eqn: 1/pi in terms of everything}). As we would point out in due course, Borwein, Borwein and Bailey used an iterative algorithm to deduce the value of $\alpha\left(58\right)$ \cite{Borwein Pi Alpha}.

First, recall the hyperbolic cosecant function, denoted by $\operatorname{csch}z$. It is defined to be \begin{align*}
    \operatorname{csch}z=\frac{1}{\operatorname{sinh}z}=\frac{2}{e^z-e^{-z}}=2\sum_{n=1}^{\infty}e^{-\left(2n-1\right)z}.
\end{align*}
The geometric series expansion is valid for $\operatorname{Re} \left(z\right)>0$ and it is particularly important. Following this, Wong \cite{wong} showed that there is a nice connection between $\csc z$ and a lattice sum (Lemma \ref{wong's lemma}).
\begin{lemma}[Wong \cite{wong}]\label{wong's lemma}
Let $r$ be a positive rational number. Then, \begin{align}
\label{log formula}
\sideset{}{'}\sum_{\left(m,n\right)\in\mathbb Z^2}\frac{\left(-1\right)^m}{m^2+rn^2}=-\frac{\pi}{\sqrt{r}}\operatorname{log}\left(2g_r^4\right),
\end{align}
where for any $d\in\mathbb{N}$, $\Sigma'$ denotes the sum over $\mathbb{Z}^d$ with the origin omitted.
\end{lemma}
\begin{proof}
Recall Ramanujan's formula for $g_r$ in (\ref{eqn: useful infinite product}). Taking logarithms on both sides, then multiplying by 4 yields  \begin{align}\label{eqn: 4 times blah blah}
    4\sum_{m=1,3,5,\ldots}\operatorname{log}\left(1-e^{-m\pi \sqrt{r}}\right)=\operatorname{log}2-\frac{\pi\sqrt{r}}{6}+4\operatorname{log} g_r = -\frac{\pi\sqrt{r}}{6} +\log\left(2g_r^4\right) .
\end{align}
As $e^{-m\pi \sqrt{r}}\in \left(0,1\right)$, we may rewrite the logarithmic term in (\ref{eqn: 4 times blah blah}) as
\begin{align}\label{eqn: log expansion}
    \operatorname{log}\left(1-e^{-m\pi \sqrt{r}}\right) = -\sum_{n=1}^{\infty} \frac{1}{n} e^{-k n \pi \sqrt{r}}.
\end{align}
Switching the order of summation in (\ref{eqn: 4 times blah blah}) after replacing the logarithmic term with (\ref{eqn: log expansion}) yields
\begin{align}\label{eqn: double sum single sum}
    -4\sum_{n=1}^\infty \sum_{m=1}^\infty \frac{1}{n} e^{-(2m-1)n \pi \sqrt{r}} = -2\sum_{n=1}^\infty \frac{\operatorname{csch}(n\pi \sqrt{r})}{n}.
\end{align}
Note that $\pi \operatorname{csch}\left(\pi z\right)$ can be expressed as the following infinite series, which can be deduced using a variety of means such as using Fourier series or the residue theorem in complex analysis: \begin{align}\label{eqn: pi csch piz}
    \pi \operatorname{csch}\left(\pi z\right)= \frac{1}{z}+\sum_{m=1}^{\infty}\frac{2z\left(-1\right)^m}{z^2+m^2}.
\end{align}
Letting $z=n\sqrt{r}$ in (\ref{eqn: pi csch piz}) and applying it to (\ref{eqn: double sum single sum}) then gives
\begin{align*}
    -2\sum_{n=1}^\infty \frac{\operatorname{csch}(n\pi \sqrt{r})}{n} &= -\frac{2}{\pi}\sum_{n=1}^\infty \left( \frac{1}{n^2 \sqrt{r}} + \sum_{m=1}^\infty \frac{2(-1)^m\sqrt{r}}{m^2+rn^2} \right)\\&= -\frac{\pi}{3\sqrt{r}} - \frac{2}{\pi} \sum_{n=1}^\infty\sum_{m=1}^\infty \frac{2(-1)^m\sqrt{r}}{m^2+rn^2}.
\end{align*}
It now suffices to prove that
\begin{align}
\label{ultima}
    -\frac{\pi}{3\sqrt{r}} - \frac{2}{\pi} \sum_{n=1}^\infty\sum_{m=1}^\infty \frac{2(-1)^m\sqrt{r}}{m^2+rn^2} = -\frac{\pi\sqrt{r}}{6} +\log(2g_r^4) .
\end{align}
We proceed by isolating the logarithmic term in (\ref{ultima}), in which we obtain
\begin{align*}
    4\sum_{n=1}^\infty\sum_{m=1}^\infty \frac{(-1)^m}{m^2+rn^2} +\frac{\pi^2}{3r}- \frac{\pi^2}{6} =   - \frac{\pi}{\sqrt{r}}\log(2g_r^4).
\end{align*}
There are a few things to notice here. That is, we have the following sums
\begin{align*}
    \sideset{}{'}\sum_{n\in \mathbb Z} \frac{1}{rn^2} = \frac{\pi^2}{3r}\quad\text{and}\quad 
    \sideset{}{'}\sum_{m\in \mathbb Z} \frac{(-1)^m}{m^2} &= -\frac{\pi^2}{6},
\end{align*}
and also
\begin{align*}
    4\sum_{n=1}^\infty\sum_{m=1}^\infty \frac{(-1)^m}{m^2+rn^2} = \sideset{}{'}\sum_{\left(m,n\right)\in\mathbb Z^2} \frac{(-1)^m}{m^2+rn^2}.
\end{align*}
So, (\ref{eqn: 4 times blah blah}) thus becomes
\begin{align*}
    \sideset{}{'}\sum_{\left(m,n\right)\in\mathbb Z^2} \frac{(-1)^m}{m^2+rn^2} + \sideset{}{'}\sum_{n\in \mathbb Z} \frac{1}{rn^2} + \sideset{}{'}\sum_{m\in \mathbb Z} \frac{(-1)^m}{m^2} = -\frac{\pi}{\sqrt{r}} \log(2g_r^4). 
\end{align*}
The left side is precisely the summation over $\mathbb Z^2$ with the origin omitted. Hence, the lemma is proven.
\end{proof}
From (\ref{log formula}), setting $r=58$, the equation of interest is \begin{align*}
    \sideset{}{'}\sum_{\left(m,n\right)\in \mathbb Z^2}\frac{\left(-1\right)^{m+1}}{m^2+58n^2}=\frac{\pi}{\sqrt{58}}\operatorname{log}\left(2g_{58}^4\right).
\end{align*}
J. Borwein, et al. described a way to decompose this double zeta sum in terms of $L$-series \cite{lattice sums}. Letting $\left(\frac{d}{n}\right)$ denote the Kronecker symbol, define \begin{align}\label{eqn: l series dirichlet}
    L_d\left(s\right)=\sum_{n=1}^{\infty}\left(\frac{d}{n}\right)\frac{1}{n^s}.
\end{align}
The Riemann zeta function is a special case of (\ref{eqn: l series dirichlet}). In particular, when $\left(\frac{d}{n}\right)=1$ yields $\zeta\left(s\right)$. Actually, the sum in (\ref{eqn: zucker robertson sum}) is mainly due to Zucker and Robertson \cite{zucker}, where they also provided a way to visualise lattice sums:  \begin{align}\label{eqn: zucker robertson sum}
    \sideset{}{'}\sum_{\left(m,n\right)\in \mathbb Z^2}\frac{\left(-1\right)^{m+1}}{\left(m^2+2Pn^2\right)^s}=\frac{\pi}{\sqrt{2P}}\operatorname{log}2+2^{1-t}\sum_{\mu\mid P}\left(1-\left(\frac{2}{\mu}\right)2^{1-s}\right)L_{\pm \mu}L_{\mp 8P/\mu},
\end{align}
where $L_{\mu}$ is taken such that $\mu\equiv \pm 1\text{ }\left(\operatorname{mod}4\right)$. Also, the $P$'s are square-free numbers which are congruent to $1 \pmod 4$ with $t$ prime factors. By defining $S_1=S_1\left(a,b,c:s\right)$ to be the series \cite{zucker} \begin{align*}
    S_1\left(a,b,c;s\right)=\sideset{}{'}\sum_{\left(m,n\right)\in \mathbb Z^2}\frac{\left(-1\right)^m}{\left(am^2+bmn+cn^2\right)^s},
\end{align*}
it remains to evaluate $S_1\left(1,0,58:1\right)$. As such, we choose $P=29$ so that (\ref{eqn: zucker robertson sum}) becomes \begin{align}\label{eqn: edwards l series}
\sideset{}{'}\sum_{\left(m,n\right)\in \mathbb Z^2}\frac{\left(-1\right)^{m+1}}{m^2+58n^2}=\frac{\pi}{\sqrt{58}}\operatorname{log}2+4L_{-8}\left(1\right)L_{29}\left(1\right).
\end{align}
Hence,
\begin{align}\label{eqn: before edwards}
    \frac{\pi}{\sqrt{58}}\operatorname{log}\left(g_{58}^4\right)=4L_{-8}\left(1\right)L_{29}\left(1\right).
\end{align}
See Theorem \ref{g invariant theorem} for a proof of (\ref{eqn: before edwards}).

As we would see in Theorem \ref{edwards main theorem on calculating l series}, Edwards \cite{edwards} provides a method to compute the values of the mentioned $L$-series in (\ref{eqn: edwards l series}). In particular, we need to compute some Dirichlet characters. We first define what a Dirichlet character is (Definition \ref{dirichlet character definition}) \cite{nist-handbook}.
\begin{definition}[Dirichlet character]\label{dirichlet character definition} Let $m>1$ be a given integer. Then, a complex-valued arithmetic function $\chi\left(n\right)$ is a Dirichlet character modulo $m$ if it is completely multiplicative, periodic with period $m$, and vanishes when $\operatorname{gcd}\left(n,m\right)>1$. In other words, Dirichlet characters modulo $m$ satisfy the three following conditions:
\begin{enumerate}[label=\textbf{(\roman*)}]
    \item \(\chi\left(mn\right) = \chi\left(m\right)\chi\left(n\right)\) and $\chi\left(1\right)=1$, i.e. $\chi$ is completely multiplicative
    \item \(\chi\left(n + m\right) = \chi\left(n\right)\), i.e. \(\chi\) is periodic with period \(m\)
    \item $\chi\left(n\right)=0$ if $\operatorname{gcd}\left(n,m\right)>1$
\end{enumerate}
The simplest possible character, called the principal character, denoted by \(\chi_0\), exists for all moduli, and is defined as follows:
\begin{align*}
\chi_1\left(n\right) = 
\begin{cases} 
  0 & \text{if } \gcd\left(n,m\right) > 1; \\
  1 & \text{if } \gcd\left(n,m\right) = 1.
\end{cases}
\end{align*}
\end{definition}
We are now in position to compute the values of the $L$-series $L_d\left(1\right)$. Note that we will discuss both positive and negative values of $d$ since we have $d=-8$ and $d=29$ as shown in (\ref{eqn: before edwards}). Theorem \ref{edwards main theorem on calculating l series} is due to Edwards \cite{edwards}.
\begin{theorem}[Edwards]\label{edwards main theorem on calculating l series}
Let $m=\left|4d\right|$ if $d$ is not congruent to 1 mod 4 and $m=\left|d\right|$ if $d\equiv 1\text{ }\left(\operatorname{mod}4\right)$. 

If $d<0$, then \begin{align*}
    L_d\left(1\right)=\frac{\pi}{m^{3/2}}\left|1+\sum_{k=2}^{m}k\chi\left(k\right)\right|.
\end{align*}
On the other hand, if $d>0$, then \begin{align}\label{eqn: edwards complicated}
    L_d\left(1\right)=\pm \frac{1}{\sqrt{m}}\operatorname{log}\left|\frac{\displaystyle\prod_{\substack{0<k<m \\ \chi(k) = 1 }}  \operatorname{sin}\left(\frac{k\pi}{m}\right)}{\displaystyle\prod_{\substack{0<k<m \\ \chi(k) = -1 }} \operatorname{sin}\left(\frac{k\pi}{m}\right)}\right|.
\end{align}
\end{theorem}
However, (\ref{eqn: edwards complicated}) is generally difficult to evaluate and it would be easier to use Dirichlet's class number formula (Theorem \ref{class number formula}) \cite{edwards}.
\begin{theorem}[Dirichlet's class number formula]\label{class number formula}
Let $E$ be the fundamental unit with norm 1. That is,
\begin{align*}
    E=\begin{cases}
        \varepsilon& \text{if } N(\varepsilon)=1\\
        \varepsilon^2& \text{if } N(\varepsilon)=-1
    \end{cases}.
\end{align*}
Also, let $h$ denote the class number of the real quadratic field $\mathbb{Q}(\sqrt{d})$. Then, we have the following class number formula for $d\equiv 1\text{ }\left(\operatorname{mod}4\right)$: \begin{align*}
    \frac{h\operatorname{log}E}{d}=L_d\left(1\right).
\end{align*}
\end{theorem}
We now prove the main result, which is Theorem \ref{g invariant theorem}.
\begin{theorem}\label{g invariant theorem} The following equation holds, where $L_d\left(1\right)$ is a Dirichlet $L$-series:
\begin{align}
\frac{\pi}{\sqrt{58}}\operatorname{log}\left(g_{58}^4\right)=4L_{-8}\left(1\right)L_{29}\left(1\right).
\end{align}
\end{theorem}
\begin{proof}
One can construct the Dirichlet character table for $m=32$. Consequently, we obtain a value for $L_{-8}\left(1\right)$. That is, $L_{-8}\left(1\right)$ has a value of
\begin{align*}
    \frac{\pi}{32^{3/2}}\left|1+3-5-7+9+11-13-15+17+19-21-23+25+27-29-31\right|
\end{align*}
which is equal to $\frac{\pi}{4\sqrt{2}}$. Also, it can be shown that \begin{align*}
    L_{29}\left(1\right)=-\frac{1}{\sqrt{29}}\operatorname{log}\left[\left(\frac{2}{5+\sqrt{29}}\right)^2\right].
\end{align*}
As mentioned previously, evaluating the quotient in (\ref{eqn: edwards complicated}) is generally very difficult. As such, we turn to Dirichlet's class number formula (Theorem \ref{class number formula}). It is known that the class number of the real quadratic field $\mathbb{Q}\left(\sqrt{29}\right)$ is 1, so $h=1$. 

We then consider the fundamental unit. Suppose we have the real quadratic field $K=\mathbb{Q}\left(\sqrt{29}\right)$. Let $\Delta$ denote the discriminant of $K$, and because $29\equiv 1\text{ }\left(\operatorname{mod}4\right)$, then $\Delta=29$. For $a,b\in\mathbb{N}$, the fundamental unit is defined to be \begin{align*}
    \frac{a+b\sqrt{\Delta}}{2}\quad\text{where}\quad \left(a,b\right)\text{ is the smallest solution to }x^2-\Delta y^2=1.
\end{align*}
This is precisely Pell's equation since 29 is non-square! The Borwein brothers mentioned that the $g$-invariant $g_{58}$ is connected to the fundamental solution of Pell's equation \cite{borwein pi proof 1987}, but this connection was not explicitly established. We have done so here. One can then use Bhāskara's method to deduce that the desired $\left(a,b\right)$ is $\left(9801,1820\right)$. One checks that \begin{align*}
    9801^2-29\cdot 1820^2=1.
\end{align*}
Recall that (\ref{eqn: what to prove ramanujan}) contains a 9801 in the denominator too! As such, we have \begin{align*}
    \frac{\operatorname{log}\left(9801+1820\sqrt{29}\right)}{3\sqrt{29}}=L_{29}\left(1\right)=-\frac{1}{\sqrt{29}}\operatorname{log}\left[\left(\frac{2}{5+\sqrt{29}}\right)^2\right]
\end{align*}
Taking the product of $L_{-8}\left(1\right)$ and $L_{29}\left(1\right)$ yields the value of the desired $g$-invariant \begin{align*}
    g_{58}=\sqrt{\frac{5+\sqrt{29}}{2}}.
\end{align*}
Moreover, this yields the following nice relation as pointed out by \cite{borwein pi proof 1987,wong}: \[g_{58}^2=u_{29}=\frac{5+\sqrt{29}}{2}. \qedhere\]
\end{proof}
In fact, the complicated quotient of trigonometric products in (\ref{eqn: edwards complicated}) can be written as \begin{align*}
\frac{\left(\sin\left(\frac{2\pi}{29}\right)\right)^{2}\left(\sin\left(\frac{3\pi}{29}\right)\right)^{2}\left(\sin\left(\frac{8\pi}{29}\right)\right)^{2}\left(\sin\left(\frac{10\pi}{29}\right)\right)^{2}\left(\sin\left(\frac{11\pi}{29}\right)\right)^{2}\left(\sin\left(\frac{12\pi}{29}\right)\right)^{2}\left(\sin\left(\frac{14\pi}{29}\right)\right)^{2}}{\left(\sin\left(\frac{\pi}{29}\right)\right)^{2}\left(\sin\left(\frac{4\pi}{29}\right)\right)^{2}\left(\sin\left(\frac{5\pi}{29}\right)\right)^{2}\left(\sin\left(\frac{6\pi}{29}\right)\right)^{2}\left(\sin\left(\frac{7\pi}{29}\right)\right)^{2}\left(\sin\left(\frac{9\pi}{29}\right)\right)^{2}\left(\sin\left(\frac{13\pi}{29}\right)\right)^{2}}
\end{align*}
which simplifies to 
\begin{align}\label{eqn: after wolfram}
    \frac{\sin^2\left(\frac{4\pi}{58}\right) \sin^2 \left( \frac{6\pi}{58} \right)}{2^{10} \sin^2\left( \frac{\pi}{58} \right)\sin^2\left( \frac{5\pi}{58} \right) \sin^2\left( \frac{7\pi}{58} \right) \sin^2\left( \frac{8\pi}{58} \right) \sin^2\left( \frac{9\pi}{58} \right) \sin^2\left( \frac{12\pi}{58} \right) \sin^2\left( \frac{13\pi}{58} \right)}.
\end{align}
However, this expression is difficult to further simplify. Having said that, by considering our proof of Theorem \ref{g invariant theorem}, we can deduce that (\ref{eqn: after wolfram}) is equal to $g_{58}^4$, which as mentioned has a nice connection to Pell's equation $x^2-29y^2=1$. It turns out that the simplification of the trigonometric quotient (\ref{eqn: after wolfram}) is not a coincidence --- it follows by using the fact that each sine product can be written as the norm of a suitable cyclotomic unit in $\mathbb{Q}\left(\zeta_{58}\right)$, where $\zeta_{58}$ denotes a primitive 58th root of unity.

Anyway, we return to the main task. By applying Lemma \ref{lemma hypergeometric} and Theorem \ref{g invariant theorem} to (\ref{eqn: 1/pi in terms of everything}), we obtain \begin{align}\label{eqn: 1/pi almost done}
    \frac{1}{\pi}&=\sum_{n=0}^{\infty}\frac{1}{256^n}\cdot \frac{\left(4n\right)!}{\left(n!\right)^2}\left[\frac{\alpha\left(58\right)}{x_{58}\left(1+k_{58}^2\right)}-\frac{\sqrt{58}}{4g_{58}^{12}}+n\sqrt{58} \frac{g_{58}^2-g_{58}^{-12}}{2}\right]x_{58}^{2n+1}.
\end{align}
By Theorem \ref{g invariant theorem}, we have \begin{align*}
    \frac{g_{58}^{12}-g_{58}^{-12}}{2}=9801.
\end{align*}
Using (\ref{eqn: x_N formula}) and writing $k'=\sqrt{1-k^2}$, we also deduce that \begin{align*}
    k_{58}=(\sqrt{2}-1)^6(13\sqrt{58}-99)\quad\text{and}\quad x_{58}=\frac{1}{9801}.
\end{align*}
Other than Theorem \ref{formula for alpha}, the Borwein brothers provided a useful formula for $\alpha\left(58\right)$ in terms of the elliptic modulus $k$ (or rather, in terms of he singular value function of the first kind $\lambda^\ast$ and the $g$-invariant \cite{pi and the agm}. In fact, Borwein, Borwein and Bailey used an iterative algorithm to deduce it \cite{Borwein Pi Alpha}. As such, we have \begin{align*}
    \alpha\left(58\right)=3g_{58}^6k_{58}(33\sqrt{29}-148).
\end{align*}
We continue putting everything together into (\ref{eqn: 1/pi almost done}) to obtain the remarkable formula in (\ref{eqn: what to prove ramanujan})! This series converges exceptionally quickly, with each term adding 8 decimal digits of accuracy \cite{wong}.
\section{Concluding Remarks}
Ramanujan's series is beautiful. We give a remark on some nice coincidences. Recall the fundamental solution to Pell's equation $x^2-29y^2=1$ that was discussed earlier, i.e. \begin{align*}
    u_{29}=\frac{5+\sqrt{29}}{2},
\end{align*}
and \begin{align*}
    u_{29}^3=70+13\sqrt{29}\quad&\text{which implies}\quad 70^2-29\cdot 13^2=-1 \\
    u_{29}^6=9801+1820\sqrt{29}\quad&\text{which implies}\quad 9801^2-29\cdot 1820^2=1
\end{align*}
Also, $2^6\left(u_{29}^6+u_{29}^{-6}\right)^2=396^4$. The number 26390 in (\ref{eqn: what to prove ramanujan}) be factorised as $29\cdot 70\cdot 13$, and looking at the big picture, we have 
\begin{align*}
    \frac{1}{\pi}=\frac{2\sqrt{2}}{9801}\sum_{n=0}^{\infty}\frac{29\cdot 70\cdot 13n+1103}{\left(n!\right)^4}\frac{\left(4n\right)!}{396^{4n}}.
\end{align*} 
This is indeed beautiful.
\section{Acknowledgements}
Thang would like to thank his co-author Wangsa for his invaluable assistance with several aspects of the proof. The former first encountered Ramanujan's remarkable formula in 2013, when he was just ten-years-old. Back then, he knew about the sigma notation and a couple of famous series for $\pi$, which are the Basel problem (or Euler $2$-series) and Leibniz's formula. They are  \begin{align*}
    \zeta\left(2\right)=\sum_{n=1}^{\infty}\frac{1}{n^2}=\frac{\pi^2}{6}\quad\text{and}\quad \sum_{n=0}^{\infty}\frac{\left(-1\right)^n}{2n+1}=\frac{\pi}{4}.
\end{align*}
respectively \cite{wong}. He somehow chanced upon (\ref{eqn: what to prove ramanujan}) one day. Although having $\pi^2$ appearing in the evaluation of $\zeta\left(2\right)$ was already unexpected, he was surprised that there exists a formula for the reciprocal of $\pi$. It was only in early 2023 where Thang decided to look into the formula again. He also wishes to express his gratitude to Prof. Bruce Berndt for his insightful correspondence and inspiring lectures during his visit to Nanyang Technological University in September 2024.

Thang would also like to thank Wong Chieh-Lei, Chan Heng Huat, and Howard Cohl for providing valuable suggestions and insightful feedback, which greatly contributed to the completion of this work.

\end{document}